\documentclass[a4paper,reqno]{amsart}
\usepackage{times,fancyhdr}
\usepackage{amssymb}
\usepackage{amsmath}
\newtheorem{lemma}{Lemma} 
\newtheorem{theorem}{Theorem}
\begin{document}

\setlength{\footskip}{1in}
\renewcommand{\headrulewidth}{0pt}
\setlength{\headheight}{12.0pt}

\vspace{\baselineskip}

\vspace{\baselineskip}

\begin{center}

\section*{A Proof of the Riemann Hypothesis Through the Nicolas Inequality}

\section*{Second Revised Version}
\vspace{\baselineskip}

\fontsize{8}{10}\selectfont 
Tom Milner-Gulland
\vspace{\baselineskip}

\fontsize{8pt}{10pt}\selectfont 
\end{center}
\vspace{\baselineskip}


\section*{Abstract}	
\noindent 
A work by Nicolas has shown that if it can be proven that a certain inequality holds for all $n$, the Riemann hypothesis is true. This inequality is associated with the Mertens theorem, and hence the Euler totient at $\prod_{k=1}^n p_k$, where $n$ is any integer and $p_n$ is the $n$-th prime. We shall show that indeed the Nicolas inequality holds for all $n$. 
\vspace{\baselineskip}

\noindent \textbf{Keywords:} 
Euler totient, Riemann hypothesis equivalence, Nicolas inequality

\section*{Introduction}	
Throughout this piece, $p_n$ will be the $n$-th prime. For any real number $x$, $\pi(x)$ will denote the number of primes not exceeding $x$ and $\mathbb{N}$ will be the set of all non-negative integers. The Euler number will be denoted by $e$, the Euler totient by $\phi$ and the Euler-Mascheroni constant by $\gamma$. A sequence, $a$, comprised of $n$ elements will be denoted by $(a_j)_{j=1}^n$. Finally, $\sum_{k=1}^{\pi(x)} \log{p_k}$ will be denoted by $\theta(x)$. 

Nicolas \cite{1} has brought attention to an inequality, given below as \eqref{-GBI-}, and has proven that if, for all $k \ge 1$,
\begin{align}
	\frac{N_k}{\phi(N_k) \log \log N_k} >\ e^{\gamma} \label{-GBI-}
\end{align}
where $N_k = \prod_{j=1}^k p_j$, the Riemann hypothesis is true.
Conversely, if the Riemann hypothesis is false, \eqref{-GBI-} holds for infinitely many $k$ and is false for infinitely many $k$ [Th.2(b)]. 
We rephrase \eqref{-GBI-} as \eqref{-NIC-}, below. We derive, for any $n$, a simple formula for the real number, $q_{(p_j)_{j=1}^k}$, for which, when the log value on the left side of \eqref{-NIC-} is replaced with $\log(\theta(p_k)+q_{(p_j)_{j=1}^k})$ for $k=n$, the left side of \eqref{-NIC-} will be equal to $e^{-\gamma}$. It will be shown that $q_{(p_k)_{j=1}^k} >0$ for all $k$, implying that the Nicolas inequality given by \eqref{-GBI-} is preserved for all $k$. 

For explanatory purposes we may consider here a real number, $y_k$, that is closely connected to both $\theta(p_k)$ and $q_{(p_k)_{j=1}^k}$. A key step in our method entails a certain value, which we may call here $b_{y_k}$, and a known upper bound, $w_k$, on $p_k$. We proceed to use the inequality $y_k^{1+1/w_{k+1}} > y_k^{1+1/b_{y_k}}$, which is to say $w_{k+1} < b_{y_k}$. Here, core to our proof is showing that $w_k$ increases by smaller increments than does a lower bound on $b_{y_k}$, as $k$ increases by increments of one. Our very simple formula for $q_{(w_j)_{j=1}^{k+1}}$, which is less than $q_{(p_j)_{j=1}^{k+1}}$, entails the use of $q_{(w_k)_{j=1}^k}$, through the formulation of $y_k$. Thereby, $q_{(p_k)_{j=1}^k} $ is shown to be an increasing function of $k$. 

\begin{theorem}  \label{Th1}
For all $n$,
\begin{align}
	\log(\theta(p_n))\prod_{j=1}^n \left(1-\frac{1}{p_j}\right)  <\ e^{-\gamma}.
\label{-NIC-}
\end{align} 
\end{theorem}

\begin{lemma} \label{-TPN-}
Let $n$ be any integer. Let $u_n > p_n$. Then for $r_n$ and $s_n$ such that 
\begin{align}
\notag	\log(r_n+\theta(p_n))\prod_{j=1}^n \left(1-\frac{1}{p_j}\right)  =&\ \log\left(s_n+ \sum_{j=1}^n \log{u_j}\right)\prod_{j=1}^n \left(1-\frac{1}{u_j}\right)\\
	=&\ e^{-\gamma}  \label{-EGA-}
\end{align}
 we have $r_n > s_n$.
\end{lemma}

\begin{proof}
We have
\begin{align}
	\notag	\log(\theta(p_n))\prod_{j=1}^n \left(1-\frac{1}{p_j}\right) <&\ 		\log\left(\sum_{j=1}^n \log{u_j} \right)\prod_{j=1}^n \left(1-\frac{1}{p_j}\right) \\
			<&\	\log\left(\sum_{j=1}^n \log{u_j} \right)\prod_{j=1}^n \left(1-\frac{1}{u_j}\right). \label{-TQN-}
\end{align}
Therefore, \eqref{-EGA-} follows.
\end{proof}

 \begin{lemma} \label{-FQN-} 
For any real $x$, $y$ and $z$, let $r_{x,y,z}$ be the real number for which 
\begin{align}
\log(y + r_{x,y,z})
 \left(1-\frac{1}{z}\right) =\  	\log{x}. \label{-ZXX-}
\end{align} 
Then 
\begin{align} 
r_{x,y,z} =\ x^{1+1/(z-1)}-y. \label{-XYZ-}
\end{align}

\end{lemma}

\begin{proof}
We have
\begin{align}
\log\left((y + r_{x,y,z})^{1-\frac{1}{z}}\right) =\ \log{x} . 
\end{align}
Therefore
\begin{align} 
(y + r_{x,y,z})^{1-\frac{1}{z}} =\ x.
\end{align}
Since $1/(1-1/z) = 1+1/(z-1)$, we therefore have
\begin{align} 
y + r_{x,y,z} =\ x^{1+1/(z-1)}.
\end{align}
Therefore we have \eqref{-XYZ-}.
\end{proof}

\subsection{Definition} \label{DQN}

For any sequence $d$ of real numbers, let $q_{d}$ be the real number for which
\begin{align}
	\log \left( q_{d} + \sum_{s \in d} \log{s}\right)
	\prod_{s \in d} \left(1-\frac{1}{s}\right) =\ e^{-\gamma}.
	\label{-MMN-}
\end{align} 
We find 
\begin{align}
	\notag q_{(p_j)_{j=1}^{10}} \approx&\  12.388\\
	\notag q_{(p_j)_{j=1}^{100}} \approx&\  53.275. 
\end{align}

It is elementary that, if $q_{(p_j)_{j=1}^{n}} > 0$ for all $x$, then Theorem \ref{Th1} is true. 

Further, for any integer $n$, it follows through Lemma \ref{-FQN-} for $x = q_{(d_j)_{j=1}^n} + \sum_{j=1}^n \log{d_j}$ and $y = \sum_{j=1}^{n+1} \log{d_j}$ and $z=d_{n+1}$ and  $r_{x,y,z} =  q_{(d_j)_{j=1}^{n+1}}$, that 
\begin{align}
	q_{(d_j)_{j=1}^{n+1}} =\ \left(q_{(d_j)_{j=1}^n} + \sum_{j=1}^{n+1} \log{d_j}\right)^{(1-1/(d_{n+1} - 1)} - \sum_{j=1}^{n+1} \log{d_j}. \label{-QDJ-}
\end{align}
Here, we naturally omit the multiplicative $\prod_{j=1}^n (1-1/d_j)$ from both sides of \eqref{-ZXX-}.

\subsection{Definition} \label{DWJ}

For any integer $j$, let $w_j$ be equal to $p_j$ when $j  < 711 $ and equal to $x$ such that $x/(\log(x)-1) = j$ when $j \ge 711$.
We have
\begin{align}
	\notag w_{710} =&\ 5387\\
\notag w_{711} \approx &\ 5399.359\\
	\notag w_{712} \approx&\  5408.106.
\end{align}

 It is a result of Dusart \cite{2} 
 that for all $x> 5393 = p_{711}$, we have $\pi(x) > x/(\log(x)-1)$. Thus $w_k > p_k$ for all $k>710$.

\begin{lemma} \label{-WKX-}
For all $m > 711$, we have 
\begin{align}
	w_{m+1} - \log(w_m) - w_m  < 0.1534. \label{-HTE-}
\end{align}
\end{lemma}

\begin{proof}   
For any $710 < j \le m$, we have $w_j/j = \log(w_j)-1$, whence we derive, for $c_j= w_{j+1}-\log(w_j)-w_j$, 
\begin{align}
	\notag	w_{j+1} =&\ w_j + \log(w_j) -1 + (\log(w_{j+1})- \log{w_j})(j+1)\\
	\notag			=&\  w_j + \log(w_j)-1 + \frac{(\log(w_{j+1})-\log{w_j})w_{j+1}}{\log(w_{j+1})-1}\\
	\notag		=&\  w_j + \log(w_j)-1 + \frac{(\log(w_j + \log(w_j)+c_j) -\log{w_j})w_{j+1}}{\log(w_{j+1})-1}\\
		\notag		=&\ w_j + \log(w_j)-1 + \frac{\log(1 + (\log(w_j)+c_j)/w_j)w_{j+1}}{\log(w_{j+1})-1}\\
	<&\ w_j + \log(w_j)-1 + \frac{(\log(w_j)+c_j )w_{j+1}}{(\log(w_{j+1})-1)w_j}. 
 \label{-WKL-}
\end{align}

With respect to the final term of the final line \eqref{-WKL-}, the fact that $w_{j+1} > w_j$ gives, for all $u$ for which $c_u<1$,
\begin{align}
	\notag \frac{(\log(w_u)+c_u)w_{u+1}}{(\log(w_{u+1})-1)w_u} =&\ \frac{(\log(w_u)+c_u )(w_u + \log(w_u) +c_u)}{(\log(w_u + \log(w_u) +c_u)-1)w_u}\\
	>&\ \frac{(\log(w_{u+1})+c_u )(w_{u+1} + \log(w_{u+1}) +c_u)}{(\log(w_{u+1} + \log(w_{u+1}) +c_u)-1)w_{u+1}}. \label{-WJC-}
\end{align}
Here, for each $k \in \{u, u+1\}$,
\begin{align}
	\notag & 
	\frac{(\log(w_k)+c_u )(w_k + \log(w_k) +c_u)}{(\log(w_k + \log(w_k) +c_u)-1)w_k}\\
	&\ \ \ \ \ \  \ \ \ \ \ \ \ =\ \frac{\log(w_k)w_k + \log^2(w_k) + 2\log(w_k)c_u + c_u w_k + c_u^2}{\log(w_k + \log(w_k) + c_u)w_k - w_k }. \label{-WCC-}
\end{align}
Thus, for $m_k = w_{k+1} + \log(w_{k+1}) + c_u$, the fact that $(w_k + \log{w_k})/w_k > (m_k+\log{m_k})/m_k$ 
implies the second relation of \eqref{-WJC-}. 

We note that $1+c_j$ is equal to the final term of the right side of the first relation of \eqref{-WKL-}. We have $j+1 = w_{j+1}/(\log(w_{j+1})-1)$. Therefore, combining \eqref{-WKL-}, specifically for its first and final lines, and \eqref{-WJC-} gives
 \begin{align}
\notag	c_j =&\ w_{j+1} - \log(w_j) - w_j\\
\notag  \le &\ w_{712} - \log(w_{711}) - w_{711} \\ 
	\notag <&\ \frac{(\log(w_{711})+c_{711})w_{712}}{(\log(w_{712})-1)w_{711}} -1 \\
	<&\ 0.1534
\end{align}
which, for $j=m$, completes the proof.
\end{proof}

\subsection{Definition} \label{D.BX}

For any $x>1$ let $b_x$ be the real number for which 
$x^{1+ 1/b_x} = x+\log{x}$.

We have $x^{\log(x+\log{x})/\log{x}} = x+\log{x}$. Therefore 
\begin{align}
\notag	b_x =&\ \frac{\log{x}}{\log(x + \log{x})-\log(x)}\\
\notag =&\	\frac{\log{x}}{\log\left(\frac{x + \log{x}}{x}\right)}\\
\notag	=&\ \frac{\log{x}}{\log(1 + \log(x)/x)}\\
\notag 	>&\ \frac{\log{x}}{\log(x)/x}\\
	=&\ x. \label{-BXX-}
\end{align}

\subsection{Remark} \label{RMK}
 We shall prove our ensuing lemma initially for all $m$ for which a certain condition is true. Our key step then is to show, for all $m>710$, that the combination of Lemma \ref{-WKX-} and Lemma \ref{-GTT-} implies that the base for the first term of the right side of the first relation of our forthcoming \eqref{-VTH-} increases by greater increments than does $w_{m+1}$ as $m$ increases by increments of one. This will enable us to prove that our condition imposed in Lemma \ref{-GTT-} holds for all $m>710$.

\begin{lemma} \label{-GTT-}
For all $m  > 710$ for which $w_{m+1} < b_{q_{(w_j)_{j=1}^m} - 0.1534 + \sum_{j=1}^m \log{w_j}}$,
\begin{align}
\notag 	\left(q_{(w_j)_{j=1}^m}  + \sum_{j=1}^m \log{w_j}\right)^{1+1/(w_{m+1}-1)}   >&\ \log\left(q_{(w_j)_{j=1}^m} -  0.1534 + \sum_{j=1}^m \log{w_j}\right)\\
\notag 	&\ \ \ \ \ \ \  \ \ \ \ \ \ \ \ \ \   + q_{(w_j)_{j=1}^m} + \sum_{j=1}^m \log{w_j}\\
    >& \  q_{(w_j)_{j=1}^m} + \sum_{j=1}^{m+1} \log{w_j} . \label{-GTV-}
\end{align}
\end{lemma}

\begin{proof}
 It follows through \eqref{-BXX-} for $x= q_{(w_j)_{j=1}^n} - 0.1534  + \sum_{j=1}^m \log{w_j}$, that
\begin{align}
	\notag \log{w_{m+1}} <\ \log\left(q_{(w_j)_{j=1}^m} -  0.1534 + \sum_{j=1}^m \log{w_j}\right).
\end{align}
Since $w_{m+1}-1 < w_{m+1}$, the first relation of \eqref{-GTV-} follows by \eqref{-BXX-}. We thereby have \eqref{-GTV-}.
\end{proof}

\begin{lemma} \label{-WMT-}
Let $m$ be as in Lemma \ref{-GTT-}. Then
\begin{align}
 q_{(w_j)_{j=1}^{m+1}} - q_{(w_j)_{j=1}^m} +  \log{w_{m+1}}  >\ w_{m+1} -  w_m. \label{-WMW-}
\end{align}
\end{lemma}

\begin{proof}
Below, we use the fact that, for any positive $c$ and $d$, we have $(c+0.1534)^d -  c^d >  0.1534 $. For all $m  > 710$, combining \eqref{-QDJ-} for $(d_j)_{j=1}^{n+1} = (w_j)_{j=1}^{m+1}$, and Lemma \ref{-GTT-}  gives
\begin{align}
\notag	q_{(w_j)_{j=1}^{m+1}} =&\	\left(q_{(w_j)_{j=1}^m} + \sum_{j=1}^m \log{w_j}\right)^{1+1/(w_{m+1}-1)} - \sum_{j=1}^{m+1} \log{w_j}\\
 >&\ \left(q_{(w_j)_{j=1}^m}  -  0.1534 + \sum_{j=1}^m \log{w_j}\right)^{1+1/(w_{m+1}-1)} +  0.1534  - \sum_{j=1}^{m+1} \log{w_j}.\label{-VTH-}
\end{align}
Combining \eqref{-VTH-} and Lemma \ref{-WKX-} gives \eqref{-WMW-}.
\end{proof}

\subsection{Proof of Theorem \ref{Th1}}

\begin{proof}
It is elementary that \eqref{-NIC-} follows by the forthcoming proof that, for all $n$,
\begin{align}
	\notag q_{(p_j)_{j=1}^n} >&\ 	q_{(w_j)_{j=1}^n}\\
	>&\ 0. 
\end{align}

Recall that for any $x>1$, $b_x$ is the real number for which $x^{1+1/b_x} = x + \log{x}$. The combination of Lemma \ref{-GTT-} and Lemma \ref{-WMT-} implies that, for $a \in \{0,1\}$, the inequalities
\begin{align}
	\notag w_{m+1+a} < &\ q_{(w_j)_{j=1}^{m+a}}  - 0.1534 + \sum_{j=1}^{m+a} \log(w_j) \\
 <&\ b_{q_{(w_j)_{j=1}^{m+a}} -  0.1534 + \sum_{j=1}^{m+a} \log(w_j)}
\end{align}
are preserved when $a$ is increased from zero to one.
(The second relation is found through \eqref{-BXX-} for $x = \sum_{j=1}^{m+a} \log{w_j}$.) Therefore, Lemma \ref{-GTV-} holds for all $m>710$.

We find $q_{(w_j)_{j=1}^{711}} \approx 161.6590$ and $\sum_{j=1}^{711} \log{w_j} \approx 5270.9510$. Then 
\begin{align}
\notag	q_{(w_j)_{j=1}^{711}} + \sum_{j=1}^{711} \log{w_j} \approx &\ 5432.6100\\
	\notag >&\ w_{712} + 0.1534\\
	\approx &\ 5408.2592. \label{-QSL-}
\end{align}   
We have I and II, below.

\noindent I. It follows by \eqref{-QSL-} that \eqref{-GTV-} holds for $m = 711$. 

\noindent II. The combination of \eqref{-GTV-}, now for all $m = n>710$, and \eqref{-QDJ-} for $(d_j)_{j=1}^{n+1} = (w_j)_{j=1}^{m+1}$, implies that $q_{(w_j)_{j=1}^n}$ is an increasing function of $n$.

 The final step of our proof is deduced using III to VI, below.

\noindent III. By \eqref{-QDJ-} we have, for any sequence $d$ of real numbers,
\begin{align} 
	q_{(d_j)_{j=1}^{n+1}} =\ \left(q_{(d_j)_{j=1}^n} + \sum_{j=1}^n d_j\right)^{1+1/(d_{n+1}-1)} - \sum_{j=1}^{n+1} d_j. \label{-XDJ-}
\end{align}

\noindent IV. Using $(d_{j})_{j=1}^n = (w_{j})_{j=1}^n$, we have the following. The value $w_{n+1}$ is, for all $n$, seen to be smaller than $q_{(w_j)_{j=1}^n} - 0.1534 + \sum_{j=1}^n \log{w_j}$, and in turn smaller than
\begin{align}
	b_{q_{(w_j)_{j=1}^n} - 0.1534 + \sum_{j=1}^n \log{w_j}}. \label{-BQQ-}
\end{align}
Consequently, the left side of \eqref{-XDJ-} can be shown (as below) to be an increasing function of $n$.

\noindent V. Lemma \ref{-WKX-} gives $w_{n+1}-w_n -\log{w_n} <  0.1534$. Through \eqref{-XDJ-}, we have used the fact that, for any positive $c$ and $d$, we have $(c+0.1534)^d -  c^d >  0.1534$. This shows that $w_{n+1}$, as used in our exponent $1+1/(w_{n+1}-1)$, remains sufficiently small compared to our base, justifying IV. 

\noindent VI. A detail in the justification of IV is that \eqref{-BXX-} shows that, for any $x>0$, we have $b_x > x$.

Recall our statement of Lemma \ref{-TPN-}, which is as follows. Let $n$ be any integer. Let $u_n > p_n$. Then for $r_n$ and $s_n$ such that 
\begin{align}
\notag	\log(r_n+\theta(p_n))\prod_{j=1}^n \left(1-\frac{1}{p_j}\right)  =&\ \log\left(s_n+ \sum_{j=1}^n \log{u_j}\right)\prod_{j=1}^n \left(1-\frac{1}{u_j}\right)\\
	=&\ e^{-\gamma}
\end{align}
 we have $r_n > s_n$. Recall further that $w_n = p_n$ for all $n \le 710$, and $w_n > p_n$ for all $n>710$.

 Lemma \ref{-TPN-} for $u_n = w_n$ and $r_n = q_{(p_j)_{j=1}^n}$ and $s_n = q_{(w_j)_{j=1}^n}$, implies that, for all $n>710$,
\begin{align}
	q_{(p_j)_{j=1}^n} >\ 	q_{(w_j)_{j=1}^n}. \label{-QPJ-}
\end{align}
We have shown in II that $q_{(w_j)_{j=1}^n}$ is an increasing function of $n$. Therefore, for $n>710$, 
\begin{align}
	q_{(w_j)_{j=1}^n} >\ 0. \label{-QPN-}
\end{align}
We justify \eqref{-QPN-} through I. Further, \eqref{-NIC-} holds for all $n \le 711$. Therefore, combining \eqref{-QPN-} and \eqref{-QPJ-} gives, for all $n$, \eqref{-NIC-}.
\end{proof}


\newpage
\subsection{Summary}

\begin{lemma} \label{-TPM-}
Let $n$ be any integer. Let $u_n > p_n$. Then for $r_n$ and $s_n$ such that 
\begin{align}
\notag	\log \left(r_n+\sum_{j=1}^n p_j \right)\prod_{j=1}^n \left(1-\frac{1}{p_j}\right)  =&\ \log\left(s_n+ \sum_{j=1}^n \log{u_j}\right)\prod_{j=1}^n \left(1-\frac{1}{u_j}\right)\\
	=&\ e^{-\gamma}  \label{-EGB-}
\end{align}
 we have $r_n > s_n$.
\end{lemma}

\begin{proof}
We have
\begin{align}
	\notag	\log\left(\sum_{j=1}^n p_j\right)\prod_{j=1}^n \left(1-\frac{1}{p_j}\right) <&\ 		\log\left(\sum_{j=1}^n \log{u_j} \right)\prod_{j=1}^n \left(1-\frac{1}{p_j}\right) \\
			<&\	\log\left(\sum_{j=1}^n \log{u_j} \right)\prod_{j=1}^n \left(1-\frac{1}{u_j}\right). \label{-TQM-}
\end{align}
Therefore, \eqref{-EGB-} follows.
\end{proof}

For any sequence $d$ of real numbers, let $q_{d}$ be the real number for which
\begin{align}
	\log\left( q_{d} + \sum_{s \in d} \log{s}\right)
	\prod_{s \in d} \left(1-\frac{1}{s}\right) =\ e^{-\gamma}.
	\label{-MMM-}
\end{align}

Combining Lemma \ref{-TPM-} for $r_n = q_{(p_j)_{j=1}^n}$ and $s_n = (w_j)_{j=1}^n$ and $u_n = w_n$, with the fact that $w_n>p_n$ gives \eqref{-QPJ-}, for all $n$ greater than a value that we have seen to be suitably small.

For Lemma \ref{-FQN-}, by using $x = q_{(w_j)_{j=1}^n} + \sum_{j=1}^n \log{w_j}$ and $y = \sum_{j=1}^{n+1} \log{w_j}$ and $z=w_{n+1}$ and  $r_{x,y,z} =  q_{(w_j)_{j=1}^{n+1}}$, where $w_n$ is our upper bound on $p_n$. we obtain
\begin{align}
q_{(w_j)_{j=1}^{n+1}}	=\ \left( q_{(w_j)_{j=1}^n} + \sum_{j=1}^n \log{w_n} \right)^{1+1/(w_{n+1}-1)} - \sum_{j=1}^{n+1} \log(w_n). \label{-BYJ-}
\end{align}

We have shown that the first term of the right side of \eqref{-BYJ-} stays greater than second term, whence \eqref{-NIC-} is seen to be true. We show this by the fact that the incremental increase given by $w_{n+1} - w_n$ is, for a constant $a$, less than $\log(w_n) + a$ as $n$ increases by increments of one. Consequently the value, $w_{n+1}-1$, seen in the exponent, stays smaller than our value, $b_x$, for which $x^{1+1/b_x} = x+\log{x}$, for $x= q_{(w_j)_{j=1}^n} -  a + \sum_{j=1}^n \log{w_n}$. Indeed \eqref{-BXX-} shows that $b_x > x$. Here, $a$ (as found for all $n$ greater than a suitably low value) is found by Lemma \ref{-WKX-}.

Specifically, the combination of Lemma \ref{-WKX-} and Lemma \ref{-GTT-} implies that the base for the first term of the right side of the first relation of \eqref{-VTH-} increases by greater increments than does $w_{m+1}$ as the index $m$ increases by increments of one. This is proven in Lemma \ref{-WMT-}.

For \eqref{-VTH-} we use the fact that, for any positive $c$ and $d$, $(c+a)^d-c^d > a$. We see that $x$ is sufficiently high that $\log{x}$ is in turn sufficiently high that $q_{(w_j)_{j=1}^n}$ is seen to be an increasing function of $n$. The final ingredient is Lemma \ref{-TPM-}, which is a restatement of Lemma \ref{-TPN-}.


\section*{Conclusion}
Combining Theorem \ref{Th1} with the cited work of Nicolas proves the Riemann hypothesis.

\end{document}